\documentclass[12pt]{amsart}
\usepackage{amssymb}
\usepackage[cmtip,all]{xy}
\usepackage{verbatim}
\usepackage{upref}
\usepackage[usenames,dvipsnames]{xcolor}
\usepackage{url}
\usepackage[normalem]{ulem}

\pdfoutput=1

\newtheorem{thm}{Theorem}[section]
\newtheorem*{thm*}{Theorem}
\newtheorem{lem}[thm]{Lemma}
\newtheorem{cor}[thm]{Corollary}
\newtheorem{prop}[thm]{Proposition}

\theoremstyle{definition} 
\newtheorem{defn}[thm]{Definition}
\newtheorem{ex}[thm]{Example}

\theoremstyle{remark} 
\newtheorem{rem}[thm]{Remark}

\newcommand{\thmref}[1]{Theorem~\textup{\ref{#1}}}
\newcommand{\corref}[1]{Corollary~\textup{\ref{#1}}}
\newcommand{\lemref}[1]{Lemma~\textup{\ref{#1}}}
\newcommand{\propref}[1]{Proposition~\textup{\ref{#1}}}

\numberwithin{equation}{section}

\newcommand{\midtext}[1]{\quad\text{#1}\quad}
\newcommand{\righttext}[1]{\quad\text{#1 }}
\renewcommand{\and}{\midtext{and}}

\newcommand{\C}{\mathbb C}
\newcommand{\T}{\mathbb T}

\newcommand{\CC}{\mathcal C}

\newcommand{\KK}{\mathcal K}
\newcommand{\II}{\mathcal I}

\newcommand{\Chi}{\raisebox{2pt}{\ensuremath{\chi}}}
\renewcommand{\epsilon}{\varepsilon}

\DeclareMathOperator{\prim}{Prim}

\DeclareMathOperator{\ad}{Ad}

\DeclareMathOperator{\ran}{ran}

\DeclareMathOperator{\infl}{Inf}

\DeclareMathOperator*{\spn}{span}
\DeclareMathOperator*{\clspn}{\overline{\spn}}

\newcommand{\id}{\text{\textup{id}}}

\newcommand{\inv}{^{-1}}

\newcommand{\variso}{\overset{\simeq}{\longrightarrow}}

\newcommand{\what}{\widehat}
\newcommand{\wilde}{\widetilde}

\renewcommand{\:}{\colon}

\renewcommand{\subset}{\subseteq}

\newcounter{nogo}

\newcommand{\cs}{\mathbf{C}^*}
\newcommand{\co}{\mathbf{Co}}

\newcommand{\ac}{\mathbf{Ac}}
\newcommand{\eqco}{\delta_G/\co}
\newcommand{\eqcom}{\eqco^m}

\newcommand{\kalg}{\KK/\cs}

\newcommand{\kac}{\KK/\ac}

\newcommand{\dst}{\textup{DSt}}

\DeclareMathOperator{\cpc}{CPC}
\DeclareMathOperator{\cpa}{CPA}

\newcommand{\fix}{\operatorname{Fix}}

\begin{document}
\title[Rigidity for $C^*$-dynamical systems, II]{Rigidity theory for $C^*$-dynamical systems and the ``Pedersen Rigidity Problem'', II}
\author[Kaliszewski]{S.~Kaliszewski}
\address{School of Mathematical and Statistical Sciences
\\Arizona State University
\\Tempe, Arizona 85287}
\email{kaliszewski@asu.edu}

\author[Omland]{Tron Omland}
\address{Department of Mathematics
\\University of Oslo
\\NO-0316 Oslo
\\Norway
\and
Department of Computer Science
\\Oslo Metropolitan University
\\NO-0130 Oslo
\\Norway}
\email{trono@math.uio.no}

\author[Quigg]{John Quigg}
\address{School of Mathematical and Statistical Sciences
\\Arizona State University
\\Tempe, Arizona 85287}
\email{quigg@asu.edu}

\date{March 29, 2019}

\thanks{The second author is funded by the Research Council of Norway through FRINATEK, project no.~240913.}


\begin{abstract}
This is a follow-up to a paper with the same title and by the same authors.
In that paper,
all groups were assumed to be abelian,
and we are now aiming to generalize the results to nonabelian groups.

The motivating point is Pedersen's theorem, which does hold for an arbitrary locally compact group $G$,
saying that two actions $(A,\alpha)$ and $(B,\beta)$ of $G$ are outer conjugate if and only if
the dual coactions $(A\rtimes_{\alpha}G,\what\alpha)$ and $(B\rtimes_{\beta}G,\what\beta)$ of $G$ are conjugate
via an isomorphism that maps the image of $A$ onto the image of $B$ (inside the multiplier algebras of the respective crossed products).

We do not know of any examples of a pair of non-outer-conjugate actions such that their dual coactions are conjugate,
and our interest is therefore exploring the necessity of latter condition involving the images;
and we have decided to use the term ``Pedersen rigid'' for cases where this condition is indeed redundant.

There is also a related problem,
concerning the possibility of a so-called equivariant coaction having a unique generalized fixed-point algebra,
that we call ``fixed-point rigidity''.
In particular, if the dual coaction of an action is fixed-point rigid, then the action itself is Pedersen rigid,
and no example of non-fixed-point-rigid coaction is known.
\end{abstract}

\maketitle

\section{Introduction}\label{intro}

Let $G$ be a locally compact group.
Given an action $\alpha$ of $G$ on a $C^*$-algebra $A$, we can form the crossed product $C^*$-algebra $A\rtimes_{\alpha}G$,
and some obvious questions to ask are:
How much does the crossed product remember of the action?
What extra information do we need in order to recover the action from the crossed product?
And what do we mean by recover, that is, what are the various types of equivalences with respect to which we can we expect to recover the action?
In general, if we only know that two crossed products are isomorphic, we cannot say much about how the corresponding actions are related.
Moreover, we think of ``rigidity'' of an action as its ability to be recovered.

Crossed-product duality refers to the problem of determining when a $C^*$-algebra is a crossed product (up to some equivalence),
and then to recover the action from the crossed product together with the dual coaction, and sometimes other data.

The first result in this direction is Imai-Takai-Takesaki duality, giving an isomorphism between $A \rtimes_\alpha G \rtimes_{\widehat{\alpha}} G$ and $A \otimes \KK(L^2(G))$,
taking the double dual action $\widehat{\widehat{\alpha}}$ to $\alpha\otimes\ad\rho$ (where $\rho$ denotes the right regular representation),
that is, recovers the action of a locally compact group up to tensoring with the compact operators.

Characterizing which $C^*$-algebras are isomorphic to a crossed product by $G$, and recovery of the action up to conjugacy was first studied by Landstad
for reduced crossed products, then for full crossed products, and later categorical versions were obtained 
(by the first and third authors).

In \cite{koqlandstad}, we study what we called outer duality, but which would be better called ``Pedersen duality'', lying in some sense between the duality theories of Takai and Landstad.
The crucial result in this regard is Pedersen's theorem,
which says that two actions $(A,\alpha)$ and $(B,\beta)$ are outer conjugate if and only if there exists an isomorphism $\Phi\colon A\rtimes_\alpha G \to B\rtimes_\beta G$ such that
\begin{gather}
\text{$\Phi$ is $\widehat{\alpha}-\widehat{\beta}$ equivariant}\label{equivariant}\\
\Phi(i_A(A))=i_B(B).\label{subalgebra}
\end{gather}
The heart of the matter is whether condition~\eqref{subalgebra} is redundant in the above result,
giving rise to the ``Pedersen rigidity problem'':
Do there exist non-outer-conjugate actions $(A,\alpha)$ and $(B,\beta)$ and an isomorphism $\Phi\colon A\rtimes_\alpha G \to B\rtimes_\beta G$ satisfying \eqref{equivariant}?

Motivated by this question, we call an action $(A,\alpha)$ \emph{Pedersen rigid} if for every other action $(B,\beta)$,
if the dual coactions $(A\rtimes_\alpha G,\widehat\alpha)$ and $(B\rtimes_\beta G,\widehat\beta)$ are conjugate,
then $(A,\alpha)$ and $(B,\beta)$ are outer conjugate. If $G$ is discrete, then every action is Pedersen rigid.
However, even when $G$ is abelian and non-discrete, the problem seems delicate.

Let $(C,\delta)$ be a coaction and $V\colon C^*(G)\to M(C)$ an equivariant homomorphism.
A related question is whether the generalized fixed-point algebra of $(C,\delta,V)$ \emph{only} depends on $C$ and $\delta$?
There are currently no examples of $V,W$ such that $C^{\gamma,V}\neq C^{\gamma,W}$.
If $(A,\alpha)$ is an action such that $(A\rtimes_\alpha G,\widehat\alpha)$ has a unique generalized fixed-point algebra, then $(A,\alpha)$ is Pedersen rigid.

Moreover, we say that a class of actions is rigid if
whenever $(A,\alpha)$ and $(B,\beta)$ are any two actions belonging to this class such that $(A\rtimes_\alpha G,\widehat\alpha)$ and $(B\rtimes_\beta G,\widehat\beta)$ are conjugate,
then $(A,\alpha)$ and $(B,\beta)$ are outer conjugate.

In \cite{koqpedersen}, we discussed Pedersen rigidity for actions of \emph{abelian} groups,
and presented several ``no-go theorems'',
that is,
situations where \eqref{subalgebra} is redundant.
For example, we showed 
that for any abelian group, the classes of all actions on commutative or stable $C^*$-algebras are both Pedersen rigid.

The goal of this paper is to generalize all the no-go theorems in \cite{koqpedersen} from abelian groups to arbitrary locally compact groups.
While some of the results in \cite{koqpedersen} carry over fully, in other cases we were only able to prove weakened versions.
For example,
we prove that every action $\alpha$ of $G$ on $A$
is strongly Pedersen rigid when
$G$ is discrete or $\alpha$ is unitary and $A$ is finite-dimensional,
or when $\alpha$ is a direct sum of strongly Pedersen rigid actions.
We also prove that the results for commutative or stable $C^*$-algebras generalize to the nonabelian case,
and that for every compact group the class of ergodic actions with full spectrum is Pedersen rigid.
In fact, the no-go theorems in the commutative or ergodic cases are stronger:
two actions are conjugate if and only if the dual coactions are.
Our proof in the compact ergodic case is significantly easier than the abelian version
(\cite[Proposition~4.8]{koqpedersen}),
due to our use of unitary eigenoperators.
In the abelian case we appealed to the cohomology of 2-cocycles.

Our no-go theorem for local rigidity
(see \thmref{no-go local})
required us to prove a new result that might be of independent interest:
$\alpha$-invariant ideals of $A$ are in one-to-one correspondence with ideals of the crossed product that are invariant for the dual coaction.
Gootman and Lazar \cite[Theorem~3.4]{gl} proved this for amenable groups,
which was enough for our abelian no-go theorem
\cite[Proposition~4.10]{koqpedersen}.
Our proof of the correspondence for arbitrary groups depends upon Landstad duality for full crossed products.

However, for one of the no-go theorems, the passing from abelian to nonabelian groups was unsuccessful.
In \cite[Theorem~4.6]{koqpedersen}
we proved that when $G$ is abelian, every 
unitary%
\footnote{A action $\alpha$ of $G$ on $A$ is \emph{unitary} if
$\alpha = \ad u$ for some strictly continuous unitary homomorphism
$u\colon G\to M(A)$.  In  \cite{koqpedersen} we used the the term ``inner'' for such
actions.}
 action of~$G$ is strongly Pedersen rigid.
The nonabelian case 
(see \corref{no-go unitary})
places a severe restriction on~$A$: it must be finite-dimensional.
This is presumably due to our method of proof --- we suspect that unitary actions are strongly Pedersen rigid in general.

Moreover, the question of whether all actions on finite-dimensional $C^*$-algebras are Pedersen rigid is also still open.

\section{Preliminaries}\label{prelim}

Throughout, $G$ will be a fixed locally compact group.
If $A$ is a $C^*$-algebra, we write $(A,\alpha)$ for an action
of $G$
and $(A,\delta)$ for a coaction of $G$.
If $A,B$ are $C^*$-algebras and
$\phi\:A\to M(B)$ is a nondegenerate homomorphism,
we use the same notation $\phi$ for the canonical extension to a unital strictly continuous homomorphism $M(A)\to M(B)$.

If $(A,\alpha)$ is an action, an \emph{$\alpha$-cocycle} is a strictly continuous unitary map $U\:G\to M(A)$ such that $U_{st}=U_s\alpha_s(U_t)$ for all $s,t\in G$.
For any $\alpha$-cocycle $U$, the composition
$\beta=\ad U\circ\alpha$ is also an action on $A$,
which is said to be \emph{exterior equivalent} to $\alpha$.
Two actions $(A,\alpha)$ and $(B,\beta)$ are \emph{outer conjugate} if $\beta$ is conjugate to an action on $A$ that is exterior equivalent to $\alpha$.

A \emph{coaction} of $G$ on $A$ is a nondegenerate faithful homomorphism $\delta\:A\to M(A\otimes C^*(G))$ such that
$(\delta\otimes\id)\circ\delta=(\id\otimes\delta_G)\circ\delta$
and
$\clspn\{\delta(A)(1\otimes C^*(G))\}=A\otimes C^*(G)$,
where
$\delta_G\:C^*(G)\to M(C^*(G)\otimes C^*(G))$ is the homomorphism determined on group elements by
$\delta_G(s)=s\otimes s$.
In particular, a coaction $\delta$ maps $A$ into
\begin{multline*}
\wilde M(A\otimes C^*(G))
=\{m\in M(A\otimes C^*(G)):
\\
m(1\otimes C^*(G))\cup (1\otimes C^*(G))m\subset A\otimes C^*(G)\}.
\end{multline*}
A coaction $(A,\delta)$ is \emph{maximal} if the canonical surjection
\[
A\rtimes_\delta G\rtimes_{\what\delta} G
\to A\otimes \KK
\]
is an isomorphism,
where we write $\KK$ to mean
the $C^*$-algebra of compact operators
$\KK(L^2(G))$.
If $(A,\alpha)$ is an action, then the dual coaction $(A\rtimes_\alpha G,\what\alpha)$ is maximal.

If $(A,\delta)$ is a coaction and $s\in G$, the \emph{$s$-spectral subspace} is
\[
A_s=\{a\in A:\delta(a)=a\otimes s\}.
\]
More generally, by nondegeneracy the coaction extends uniquely to a homomorphism, still denoted by $\delta$, from $M(A)$ to $M(A\otimes C^*(G))$, and we have spectral subspaces for these too:
\[
M(A)_s=\{m\in M(A):\delta(m)=m\otimes s\}.
\]
However, in general the extended map $\delta\:M(A)\to M(A\otimes C^*(G))$ is not a coaction,
because we may have $\delta(M(A))\not\subset \wilde M(M(A)\otimes C^*(G))$.
The \emph{fixed-point algebra} of $A$ under $\delta$ is
\[
A^\delta=A_e=\{a\in A:\delta(a)=a\otimes 1\},
\]
where $e$ is the identity element of $G$.

If $(A,\alpha)$ and $(B,\beta)$ are actions,
then a homomorphism $\phi\:A\to B$ is \emph{$\alpha-\beta$ equivariant} if
$\phi\circ\alpha_s=\beta_s\circ\phi$ for each $s\in G$.
On the other hand, if $(A,\delta)$ and $(B,\epsilon)$ are coactions,
then a homomorphism $\phi\:A\to B$ is \emph{$\delta-\epsilon$ equivariant} if
the following diagram commutes:
\[
\xymatrix{
A \ar[r]^-\delta \ar[d]_\phi
&\wilde M(A\otimes C^*(G)) \ar[d]^{\phi\otimes\id}
\\
B \ar[r]_-\epsilon
&\wilde M(B\otimes C^*(G)).
}
\]
Note that
the properties of the ``tilde multiplier algebras'' such as
$\wilde M(A\otimes C^*(G))$ guarantee that
the right-hand vertical homomorphism $\phi\otimes\id$ is well-defined, even though $\phi\:A\to B$ may be degenerate
(see, for example,
\cite[discussion following Definition~3.2]{klqfunctor}).

If $(A,\delta)$ is a coaction, an ideal $I$ of $A$ is \emph{strongly $\delta$-invariant} if
\[
\clspn\{\delta(I)(1_{M(A)}\otimes C^*(G))\}=I\otimes C^*(G),
\]
in which case $\delta$ restricts to a coaction $\delta_I$ on $I$,
which is maximal if $\delta$ is.
Moreover, the inclusion map $\iota\:I\hookrightarrow A$ is $\delta_I-\delta$ equivariant, and the crossed product $\iota\rtimes G$ maps
$I\rtimes_{\delta_I} G$
faithfully onto an ideal of $A\rtimes_\delta G$.
We identify $I\rtimes_{\delta_I} G$ with this ideal.
Finally, $\delta$ descends to a coaction $\delta^I$ on $A/I$,
which is maximal if $\delta$ is,
and
this gives a short exact sequence
\[
\xymatrix{
0 \ar[r]
&I\rtimes_{\delta_I} G \ar[r]
&A\rtimes_\delta G \ar[r]
&(A/I)\rtimes_{\delta^I} G \ar[r]
&0
}
\]
that is equivariant for the dual actions, by
\cite[Theorem~2.3]{nil:full}.

An \emph{equivariant coaction} is a triple $(A,\delta,V)$,
where $(A,\delta)$ is a coaction
and $V\:C^*(G)\to M(A)$ is a $\delta_G-\delta$ equivariant nondegenerate homomorphism.
The \emph{generalized fixed-point algebra} associated to an equivariant 
coaction $(A,\delta,V)$ is the set $A^{\delta,V}$ of all $m\in M(A)$ satisfying \emph{Landstad's conditions}
\begin{enumerate}
\item $\delta(m)=m\otimes 1$;
\item $mV(f),V(f)m\in A$ for all $f\in C_c(G)$;
\item $s\mapsto \ad V_s(m)$ is norm continuous.
\end{enumerate}
Note that (1) says that $m\in M(A)_e$,
and (3) says that $\ad V$ is an action on $A^{\delta,V}$.
If $(A,\delta,V)$ is an equivariant maximal coaction (so $\delta$ is maximal),
and if we let $B=A^{\delta,V}$ and $\alpha=\ad V\:G\curvearrowright B$,
Landstad duality for full crossed products \cite[Theorem~3.2]{kqfulllandstad}
says that there is an isomorphism
\[
(A^{\delta,V}\rtimes_\alpha G,\what\alpha)\variso (A,\delta)
\]
taking $i_G$ to $V$ and $i_B$ to the inclusion $B\hookrightarrow M(A)$.

A \emph{$\KK$-algebra}
is a pair $(A,\iota)$,
where $A$ is a $C^*$-algebra
and $\iota\:\KK\to M(A)$ is a nondegenerate homomorphism.
The \emph{relative commutant} of a $\KK$-algebra $(A,\iota)$ is the $C^*$-algebra
\[
C(A,\iota)=\{m\in M(A):m\iota(k)=\iota(k)m\text{ for all }k\in\KK\}.
\]
The \emph{canonical isomorphism} $\theta_A\:C(A,\iota)\otimes \KK\variso A$ is determined on elementary tensors by
$\theta_A(a\otimes k)=a\iota(k)$.

By \cite[Lemma~3.8]{klqfunctor2}, if $(A,\iota)$ and $(B,\jmath)$ are $\KK$-algebras and $\phi\:A\to B$ is a homomorphism such that
\[
\phi(a\iota(k))=\phi(a)\jmath(k)\righttext{for all}a\in A,k\in\KK,
\]
then there is a unique homomorphism $C(\phi)\:C(A,\iota)\to C(B,\jmath)$ such that
\[
C(\phi)(a\iota(k))=\phi(a)\jmath(k)\righttext{for all}a\in C(A,\iota),k\in\KK.
\]
Again, the subtlety is that, even though $\phi$ might be degenerate, we are extending part of the way into $M(A)$.
To belabor the point: we cannot express the condition on $\phi$ in the form $\phi\circ\iota=\jmath$, because we do not require the homomorphism $\phi\:A\to B$ to be nondegenerate, and consequently we have no right to expect that it will extend to a homomorphism $M(A)\to M(B)$.

A \emph{$\KK$-action} is a triple $(A,\alpha,\iota)$,
where $(A,\alpha)$ is an action
and $(A,\iota)$ is a $\KK$-algebra
such that $\alpha$ is trivial on $\iota(\KK)$.
In this case $\alpha$ restricts to an action
$C(\alpha)$ on $C(A,\iota)$.

We adapt a few concepts from \cite{koqpedersen} from abelian to arbitrary locally compact groups $G$.

\begin{defn}
A maximal coaction $(A,\delta)$ of $G$ is \emph{strongly fixed-point rigid} if it has a unique generalized fixed-point algebra,
i.e., for any two $G$-equivariant strictly continuous unitary homomorphisms $V,W\:G\to M(A)$ we have
\[
A^{\delta,V}=A^{\delta,W}.
\]
An action of $G$ is \emph{strongly Pedersen rigid} if its dual coaction is strongly fixed-point rigid.
\end{defn}

\begin{defn}
A maximal coaction $(A,\delta)$ of $G$ is \emph{fixed-point rigid} if
the automorphism group of $(A,\delta)$ acts transitively on the set of generalized fixed-point algebras,
i.e., for any two $\what\delta$-equivariant strictly continuous unitary homomorphisms $V,W\:G\to M(A)$ there is an automorphism $\Theta$ of $(A,\delta)$ such that
\[
\Theta(A^{\delta,V})=A^{\delta,W}.
\]
An action of $G$ is \emph{Pedersen rigid} if its dual coaction is fixed-point rigid.
\end{defn}

The elementary theory of \cite[Section~3]{koqpedersen}
carries
over to the case of nonabelian $G$;
in particular,
strong Pedersen rigidity of
an action $(B,\alpha)$ of $G$ is 
equivalent to the following:
for every action $(C,\beta)$ of $G$,
if $\Theta\:(B\rtimes_{\alpha}G,\what\alpha)\variso (C\rtimes_\beta G,\what\beta)$ is a conjugacy then $\Theta(i_B(B))=i_C(C)$,
and
Pedersen rigidity of
$(B,\alpha)$ of $G$ is 
equivalent to the following:
for every action $(C,\beta)$ of $G$,
$\alpha$ and $\beta$ are outer conjugate if and only if
the dual coactions $(B\rtimes_\alpha G,\what\alpha)$ and $(C\rtimes_\beta G,\what\beta)$ are conjugate.
Moreover
(and this wasn't explicitly mentioned in \cite{koqpedersen})
both strong fixed-point rigidity and fixed--point rigidity are preserved by conjugacy of coactions,
and consequently
both strong Pedersen rigidity and Pedersen rigidity are preserved by outer conjugacy of actions.

\theoremstyle{plain}
\newtheorem*{sfpr*}{Strong Fixed-Point Rigidity Problem}
\begin{sfpr*}
Is every maximal coaction strongly fixed-point rigid?
\end{sfpr*}

Equivalently:
\theoremstyle{plain}
\newtheorem*{spr*}{Strong Pedersen Rigidity Problem}
\begin{spr*}
Is every action strongly Pedersen rigid?
\end{spr*}

\theoremstyle{plain}
\newtheorem*{fpr*}{Fixed-Point Rigidity Problem}
\begin{fpr*}
Is every maximal coaction fixed-point rigid?
\end{fpr*}

Equivalently:
\theoremstyle{plain}
\newtheorem*{pr*}{Pedersen Rigidity Problem}
\begin{pr*}
Is every action Pedersen rigid?
\end{pr*}

In \cite{koqpedersen} we proved a number of no-go theorems,
each giving particular sufficient conditions for a positive answer to 
the Pedersen Rigidity Problem.
Some of these are phrased in terms of the following:

\begin{defn}
A class $\CC$ of actions is \emph{Pedersen rigid} if any two actions $(A,\alpha)$ and $(B,\beta)$ in $\CC$ are outer conjugate if and only if the dual coactions
$(A\rtimes_\alpha G,\what\alpha)$ and $(B\rtimes_\beta G,\what\beta)$ are conjugate.

\end{defn}

\section{Discrete groups}\label{discrete}

\begin{thm}\label{no-go discrete}
If $G$ is discrete, then every action of $G$ is strongly Pedersen rigid.
\end{thm}

\begin{proof}
Let $(A,\alpha)$ be an action.
Since $G$ is discrete,
$i_A(A)$ is the fixed-point algebra of $\what\alpha$
(see \lemref{discrete fixed} below),
and this must coincide with all general fixed-point algebras.
\end{proof}

The following lemma is presumably folklore, but since we could not find a reference we include proof.
\begin{lem}\label{discrete fixed}
Let $(A,\delta)$ be a maximal coaction of a discrete group $G$.
Then for every equivariant homomorphism $V\:G\to M(A)$ we have
$A^{\delta,V}=A^\delta$,
where as usual
\[
A^\delta=A_e=\{a\in A:\delta(a)=a\otimes 1\}.
\]
\end{lem}

\begin{proof}
Let $V$ be an equivariant homomorphism.
First note that $A^{\delta,V}\subset A$ because if $m\in A^{\delta,V}$ then
\[
m=mV(1_{C^*(G)})\in A.
\]
On the other hand, if $a\in A^\delta$ then $a\in M(A)_e$,
for every $c\in C^*(G)$ we have $ac,ca\in A$ because $c\in M(A)$,
and $s\mapsto \ad V_s(a)$ is trivially norm continuous by discreteness of $G$.
Therefore $A^\delta\subset A^{\delta,V}$.
\end{proof}

\section{Stable $C^*$-algebras}\label{stable}

\begin{thm}\label{no-go stable}
The class of actions on stable $C^*$-algebras possessing strictly positive elements is Pedersen rigid.
\end{thm}

\begin{proof}
The proof of \cite[Proposition~4.2]{koqpedersen} carries over verbatim, since the quoted result \cite[Section~8 Proposition]{com} has no restriction on the group.
\end{proof}

\begin{rem}
There is an error in the discussion following \cite[Proposition~4.2]{koqpedersen}, where we said that we won't
``find any examples of multiple generalized fixed-point algebras unless at least one of $A$ and $B$ is nonstable''.
This seems to be making an assertion about strong fixed-point rigidity,
whereas the proposition only concerns fixed-point rigidity (when phrased in terms of coactions).
The discussion should be changed to something along the following lines:
delete the first sentence ``Thus (assuming,\dots)'', since it is not obvious how to rephrase it in a useful way in terms of Pedersen rigidity,
then in the second sentence change
``phenomenon of multiple generalized fixed-point algebras''
to
``phenomenon of conjugate dual coactions of non-outer conjugate actions''.
\end{rem}

\section{Commutative $C^*$-algebras}\label{commutative}

\begin{thm}\label{no-go commutative}
Actions on commutative $C^*$-algebras are conjugate if and only if the dual coactions are conjugate.
In particular, the class of 
such
actions 
is Pedersen rigid.
\end{thm}

\begin{proof}
The proof of \cite[Proposition~4.3]{koqpedersen} carries over verbatim,
since it did not use the standing hypothesis from that paper that $G$ be abelian.
\end{proof}

\section{Compact groups}\label{compact sec}

In
\cite[theorem~8]{lancompact},
Landstad proves that if $G$ is a compact group and $(A,G,\alpha)$ is an ergodic action
with \emph{full spectrum} (meaning that every $\pi\in \what G$ occurs in $\alpha$ with multiplicity $\dim \pi$),
then there is a \emph{unitary eigenoperator}
$U\in M(A\otimes \KK(L^2(G)))$,
i.e.,
\begin{equation}\label{eigen}
(\alpha_s\otimes\id)(U)=U(1\otimes \rho_s)\righttext{for all}s\in G.
\end{equation}
here $\rho$ is the right regular representation of $G$.
From now on we will write $\KK=\KK(L^2(G))$.

\begin{prop}\label{id}
If $G$ is compact and $(A,G,\alpha)$ is an ergodic action with full spectrum, then
\[
(A\otimes \KK,\alpha\otimes\ad\rho)\simeq (A\otimes \KK,\alpha\otimes\id).
\]
\end{prop}

\begin{proof}
Let $U$ be a unitary eigenoperator as in \eqref{eigen}.
Then for all $y\in A\otimes \KK$,
\begin{align*}
\ad U\circ (\alpha_s\otimes\ad\rho_s)(y)
&=U(\id\otimes \ad\rho_s)\bigl((\alpha_s\otimes\id)(y)\bigr)U^*
\\&=U(1\otimes\rho_s)(\alpha_s\otimes\id)(y)(1\otimes\rho_s^*)U^*
\\&=(\alpha_s\otimes\id)(U)(\alpha_s\otimes\id)(y)\bigl(U(1\otimes\rho_s)\bigr)^*
\\&=(\alpha_s\otimes\id)(Uy)\bigl((\alpha_s\otimes\id)(U)\bigr)^*
\\&=(\alpha_s\otimes\id)(UyU^*
\\&=(\alpha_s\otimes\id)\circ\ad U(y).
\qedhere
\end{align*}
\end{proof}

\begin{lem}\label{fixed}
If $G$ is compact and $(A,G,\alpha)$ is an ergodic action with full spectrum, then
\[
(A\otimes\KK)^{\alpha\otimes\id}=1_A\otimes \KK.
\]
\end{lem}

\begin{proof}
Let $\omega$ be the unique $G$-invariant state on $A$, so that for all $a\in A$,
\[
\omega(a)1_A=\int_G\alpha_s(a)\,ds.
\]
Then for all $a\in A$, $T\in \KK$,
\begin{align*}
\int_G(\alpha_s\otimes\id)(a\otimes T)\,ds
&=\int_G \alpha_s(a)\,ds\otimes T
\\&=\omega(a)1_A\otimes T
\\&=\omega(a)(1_A\otimes T).
\end{align*}
Thus
\[
1_A\otimes\KK\subset (A\otimes\KK)^{\alpha\otimes\id},
\]
and on the other hand, by linearity, density, and continuity,
\[
(A\otimes\KK)^{\alpha\otimes\id}\subset 1_A\otimes\KK.
\]
\end{proof}

\begin{thm}\label{rigid}
Let $G$ be compact, and let $(A,\alpha)$ and $(B,\beta)$ be ergodic actions of $G$ with full spectrum. If $\what\alpha\simeq \what\beta$, then $\alpha\simeq \beta$.
In particular, the class of ergodic actions of $G$ with full spectrum is Pedersen rigid.
\end{thm}

\begin{proof}
Since
\[
(A\rtimes_\alpha G,\what\alpha)\simeq (B\rtimes_\beta G,\what\beta),
\]
we have
\[
(A\rtimes_\alpha G\rtimes_{\what\alpha} G,\what{\what\alpha})\simeq
(B\rtimes_\beta G\rtimes_{\what\beta} G,\what{\what\beta}),
\]
so by crossed-product duality
\[
(A\otimes\KK,\alpha\otimes\ad\rho)
\simeq (B\otimes\KK,\beta\otimes\ad\rho),
\]
and hence by \propref{id} we have an isomorphism
\[
\theta\:(A\otimes\KK,\alpha\otimes\id)
\simeq (B\otimes\KK,\beta\otimes\id).
\]
Then by \lemref{fixed},
\[
\theta(1_A\otimes\KK)
=\theta\bigl((A\otimes\KK)^{\alpha\otimes\id}\bigr)
=(B\otimes\KK)^{\beta\otimes\id}
=(1_B\otimes\KK).
\]
Thus
\[
\theta\:A\otimes\KK\variso B\otimes\KK
\]
is a $\KK$-isomorphism,
so by \cite[Theorem~4.4]{koqstable}
$\theta$ preserves the relative commutants:
\[
\theta(A\otimes 1_{B(L^2(G))})=B\otimes 1_{B(L^2(G))}.
\]
Thus $\theta$ induces an equivariant isomorphism
\[
\theta_0\:(A,\alpha)\variso (B,\beta).
\qedhere
\]
\end{proof}

\section{Categories and functors}

In preparation for our no-go theorem on local rigidity
(\thmref{no-go local}),
we recall some categorical machinery from \cite{koqmaximal}.
The category $\cs$ has $C^*$-algebras as objects,
and the morphisms are just the usual homomorphisms between $C^*$-algebras (not into multiplier algebras).
The category $\ac$ of actions has actions $(A,\alpha)$ as objects,
and 
a morphism $\phi\:(A,\alpha)\to (B,\beta)$ is an $\alpha-\beta$ equivariant homomorphism $\phi\:A\to B$. Note that we are not allowing $\phi$ to take values in the multiplier algebra $M(B)$, since this would make it inconvenient to handle ideals.
Warning: in earlier papers we used the same notation for categories in which the morphisms were nondegenerate homomorphisms into multiplier algebras; the appropriate choice depends upon the context.

The category $\co$ of coactions has coactions $(A,\delta)$ as objects,
and a morphism $\phi\:(A,\delta)\to (B,\epsilon)$ is a $\delta-\epsilon$ equivariant homomorphism $\phi\:A\to B$.
The full subcategory of maximal coactions is denoted by $\co^m$.

The category $\eqco$ of equivariant actions has equivariant coactions $(A,\delta,V)$ as objects,
and a morphism $\phi\:(A,\delta,V)\to (B,\epsilon,W)$ is a morphism $\phi\:(A,\delta)\to (B,\epsilon)$ such that $\phi\circ V=W$.
The full subcategory of maximal equivariant coactions,
where the coactions are required to be maximal,
is denoted by $\eqco^m$.

The category $\kalg$ of $\KK$-algebras has $\KK$-algebras as objects,
and a morphism $\phi\:(A,\iota)\to (B,\jmath)$
is a homomorphism $\phi\:A\to B$ such that
\[
\phi(a\iota(k))=\phi(a)\jmath(k)\righttext{for all}a\in A,k\in\KK.
\]

The category $\kac$ of $\KK$-actions has $\KK$-actions as objects,
and a morphism $\phi\:(A,\alpha,\iota)\to (B,\beta,\jmath)$
is a morphism $\phi\:(A,\alpha)\to (B,\beta)$ in $\ac$
such that
$\phi\:(A,\iota)\to (B,\jmath)$ is a morphism in $\kalg$.

The \emph{destabilization functor} $\dst\:\kac\to \ac$ is given by
\begin{align*}
\dst(A,\alpha,\iota)&=\bigl(C(A,\iota),C(\alpha)\bigr)\\
\dst(\phi)&=C(\phi).
\end{align*}

The categorial Landstad duality theorem
for actions
\cite[Theorem~5.1]{clda}
(see also \cite[Theorem~2.2]{koqlandstad})
can be formulated as follows:
the functor
\[
\cpa\:\ac\to \eqcom
\]
defined by
\begin{align*}
\cpa(A,\alpha)&= (A\rtimes_\alpha G,\what\alpha,i_G)\\
\cpa(\phi)&= \phi\rtimes G
\end{align*}
is an equivalence.

Given an equivariant coaction $(A,\delta,V)$,
the homomorphism
\[
u^A:=j_A\circ V\:G\to M(A\rtimes_\delta G)
\]
is a $\what\delta$-cocycle,
and we write the perturbed action on $A\rtimes_\delta G$ as
\[
\wilde\delta:=u^A\circ \what\delta.
\]
The functor $\cpc\:\delta_G/\co^m\to \kac$
is given on objects by
\begin{align*}
\cpc(A,\delta,V)&=(A\rtimes_\delta G,\wilde\delta,V\rtimes G),
\end{align*}
and if $\phi\:(A,\delta,V)\to (B,\epsilon,W)$ is a morphism then
\[
\cpc(\phi)\:(A\rtimes_\delta G,\wilde\delta,V\rtimes G)
\to (B\rtimes_\epsilon D,\wilde\epsilon,W\rtimes G)
\]
is the morphism in $\kac$ given by
$\cpc(\phi)=\phi\rtimes G$.

The quasi-inverse functor $\fix$ is determined by the commutative diagram
\[
\xymatrix{
\eqcom \ar[r]^-{\cpc} \ar[dr]_{\fix}
&\kac \ar[d]^{\dst}
\\
&\ac
}
\]
Given an equivariant maximal coaction $(A,\delta,V)$, we write
\begin{align*}
\fix A&=C(A\rtimes_\delta G,V\rtimes G)\\
\fix\delta&=C(\wilde\delta),
\end{align*}
so that
\[
\fix(A,\delta,V)=(\fix A,\fix\delta).
\]
If $\phi\:(A,\delta,V)\to (B,\epsilon,W)$ is a morphism in $\eqcom$ then
\[
\fix(\phi)\:(\fix A,\fix\delta)\to (\fix B,\fix\epsilon)
\]
is the morphism in $\ac$ given by
\[
\fix(\phi)=C(\phi\rtimes G).
\]
Categorical Landstad duality for actions can be illustrated by the commutative diagram
\[
\xymatrix{
\ac \ar[r]^-{\cpa} \ar[d]_\simeq
&\eqcom \ar[d]^{\cpc} \ar[dl]^{\fix}
\\
\ac
&\kac \ar[l]^{\dst}
}
\]
of functors.

We will need to know that the functor $\fix$ is exact,
and we prove this in \lemref{exact lem} below.
To be clear: when we refer to a short exact sequence in any of our categories in which the objects are $C^*$-algebras with extra structure
and the morphisms are homomorphisms that preserve the structure,
we mean that we have a sequence of morphisms in the category
such that the homomorphisms give a short exact sequence of $C^*$-algebras.

\begin{lem}\label{exact lem}
The functor $\fix$ is exact.
\end{lem}

\begin{proof}
By construction, it suffices to verify that the two functors $\cpc$ and $\dst$ are exact.
The first follows from exactness of the functor
\[
(A,\delta)\mapsto A\rtimes_\delta G
\]
from $\co$ to $\cs$, which is proven in \cite[Theorem~2.3]{nil:full}.

For $\dst$, it suffices to show that the functor
\[
(A,\iota)\mapsto C(A,\iota)
\]
from $\kalg$ to $\cs$ is exact.
This is presumably folklore, but we include the argument for completeness:
let
\[
\xymatrix{
0 \ar[r]
&(I,\rho) \ar[r]^-\psi
&(A,\iota) \ar[r]^-\pi
&(Q,\jmath) \ar[r]
&0
}
\]
be a short exact sequence of $\KK$-algebras.
by naturality of the isomorphisms $\theta$ from destabilization,
the sequence
\[
\xymatrix@C+20pt@R-20pt{
0 \ar[r]
&C(I,\rho)\otimes\KK \ar[r]^-{C(\phi)\otimes\id_\KK}&{}
\\
&C(A,\iota)\otimes\KK \ar[r]^-{C(\pi)\otimes\id_\KK}&{}
\\
&C(Q,\jmath)\otimes\KK \ar[r]
&0
}
\]
is exact.
Abstracting this,
it now suffices to show why a sequence
\[
\xymatrix{
0 \ar[r]
&J \ar[r]^-\psi
&B \ar[r]^-\tau
&R \ar[r]
&0
}
\]
in $\cs$ must be exact if the sequence
\[
\xymatrix{
0 \ar[r]
&J\otimes\KK \ar[r]^-{\psi\otimes\id}
&B\otimes\KK \ar[r]^-{\tau\otimes\id}
&R\otimes\KK \ar[r]
&0
}
\]
is exact.
Since $\KK$ is nuclear,
\[
\ker(\tau\otimes\id)=(\ker\tau)\otimes \KK.
\]
Also,
\[
\ran(\psi\otimes\id)=(\ran\psi)\otimes \KK.
\]
Since $\ker(\tau\otimes\id)=\ran(\psi\otimes\id)$ by assumption,
we must have
\[
\ker\tau=\ran\psi,
\]
as desired.
\end{proof}

\section{Local rigidity}\label{local sec}

In this section we will prove the following generalization of
\cite[Proposition~4.10 and Corollary~4.12]{koqpedersen} from abelian to arbitrary locally compact groups:

\begin{thm}\label{no-go local}
Let $(A,\alpha)$ be an action,
and let $\II$ be a family of $\alpha$-invariant ideals of $A$ such that $A=\clspn \II$.
If for each $I\in\II$ the restricted action $\alpha^I$ is strongly Pedersen rigid, then $\alpha$ is strongly Pedersen rigid.
\end{thm}

\begin{proof}
We only need one modification of the proof of \cite[Proposition~4.10]{koqpedersen}: instead of referring to \cite[Theorem~3.4]{gl},
we use \thmref{ideal} below instead.
\end{proof}

The above proof
rests upon the following
correspondence between invariant ideals of an action and of the dual coaction.
It is proved for amenable $G$ in \cite[Theorem~3.4]{gl}.

\begin{thm}\label{ideal}
For any action $(A,\alpha)$,
the assignment
$I\mapsto I\rtimes G$
gives
a one-to-one correspondence
between $\alpha$-invariant ideals of $A$ and 
strongly $\what\alpha$-invariant ideals of $A\rtimes_\alpha G$.
\end{thm}

\begin{proof}
First,
let $K$ be a strongly $\what\alpha$-invariant ideal of $A\rtimes_\alpha G$.
Then
the elementary \lemref{restrict} below gives us an equivariant maximal coaction
\[
(K,\what\alpha^K,(i_G)_K).
\]
Then we have a short exact sequence
in $\eqcom$,
so by \lemref{exact lem} we can apply the functor $\fix$ to get a short exact sequence
\begin{equation}\label{LBR}
\xymatrix{
0 \ar[r]
&(L,\gamma) \ar[r]
&(B,\beta) \ar[r]
&(R,\sigma) \ar[r]
&0
}
\end{equation}
in $\ac$.
The natural equivariant isomorphism $\theta\:B\variso A$
takes the ideal $L$ to an $\alpha$-invariant ideal $I$ of $A$.
Let $\mu$ be the restriction of the action $\alpha$ to $I$.
Then $I\rtimes_\mu G$ is an ideal of $A\rtimes_\alpha G$,
and it remains to show that
\[
I\rtimes_\mu G=K.
\]
Applying the crossed-product functor $\cpa$ to the short exact sequence \eqref{LBR}
gives a $\what\beta$-invariant ideal $L\rtimes_\gamma G$ of $B\rtimes_\beta G$.
Then applying the natural isomorphism
\[
\Theta\:\cpa\circ \fix\variso \id,
\]
we have
\[
\Theta(L\rtimes G)=K.
\]
On the other hand,
by the category equivalence we have $\Theta=\theta\rtimes G$,
so
\begin{align*}
I\rtimes G
&=\theta(L)\rtimes G
\\&=\Theta(L\rtimes G)
\\&=K.
\end{align*}

We turn to the uniqueness.
Suppose that 
$I$ and $J$ are $\alpha$-invariant ideals of $A$ such that
\[
I\rtimes G=J\rtimes G=K.
\]
Applying the natural isomorphism
$\theta\:\fix\circ\cpa\variso \id$,
we get
\[
I=\theta(\fix K)=J.
\qedhere
\]
\end{proof}

\begin{lem}\label{restrict}
Let $(B,\delta,V)$ be an equivariant coaction,
and let $K$ be a strongly $\delta$-invariant ideal of $B$.
Then there is a unique $\delta_G-\delta_K$ equivariant nondegenerate homomorphism
\[
V_K\:C^*(G)\to M(K)
\]
such that
for all $c\in C^*(G)$ and $k\in K$ we have
\begin{equation}\label{VK}
V_K(c)k=V(c)k.
\end{equation}
\end{lem}

\begin{proof}
Let $\sigma\:B\to M(K)$ be the canonical nondegenerate homomorphism given by
\[
\sigma(b)k=bk\righttext{for}b\in B,k\in K.
\]
Define $V_K$ by the commutative diagram
\[
\xymatrix{
C^*(G) \ar[r]^-V \ar[dr]_{V_K}
&M(B) \ar[d]^\sigma
\\
&M(K).
}
\]
Then \eqref{VK} holds because
\[
V_K(c)k
=\sigma(V(c))k
=V(c)k.
\qedhere
\]
\end{proof}

\subsection*{Bundles}\label{bundle subsec}

In \cite[Corollary~4.12]{koqpedersen} we proved that
when $G$ is abelian
every locally unitary action on a continuous trace $C^*$-algebra is strongly Pedersen rigid.
This followed immediately upon combining
\cite[Propositions~4.6 and 4.10]{koqpedersen},
on unitary actions and inductive limits of actions, respectively.
While
\thmref{no-go local}
is a fully functional generalization of
\cite[Proposition~4.10]{koqpedersen}
to nonabelian groups,
\corref{no-go unitary} below is only a partial generalization of
\cite[Proposition~4.6]{koqpedersen},
restricting to
unitary actions on finite-dimensional $C^*$-algebras.
Consequently, the best we can do toward locally unitary actions
of arbitrary groups is
pointwise unitary actions on direct sums of finite-dimensional algebras,
because we need the ideals of the bundle $C^*$-algebra determined
by neighborhoods in the base space
to be finite-dimensional.
But in fact this case is so special that it is not really a bundle result at all:

\begin{cor}\label{no-go bundle}
Every direct sum of strongly Pedersen rigid actions is strongly Pedersen rigid.
\end{cor}

This is of course a special case of \thmref{no-go local}.

\section{Finite-dimensional algebras and unitary actions}\label{unitary sec}

\begin{thm}\label{sufficient}
Let $(A,\delta)$ be a maximal coaction.
Suppose that there is an equivariant homomorphism $V\:G\to M(A)$ such that
the generalized fixed-point algebra
$B=A^{\delta,V}$
satisfies:
\begin{enumerate}
\item $B=M(A)_e$, and

\item the norm topology on $B$ equals the relative strict topology from $M(A)$.

\end{enumerate}
Then $(A,\delta)$ is strongly fixed-point rigid.
\end{thm}

\begin{proof}
Let $W\:G\to M(A)$ be another equivariant 
homomorphism,
and let $C=A^{\delta,W}$.
We must show that $C=B$.
First note that $C\subset M(A)_e=B$.

Define $U\:G\to M(A)$ by
\[
U_s=W_sV_s^*.
\]
Since $V,W$ are strictly continuous and bounded, $U$ is also strictly continuous.
The equivariance of $V,W$
trivially implies by direct computation
that
for all $s\in G$ we have
\[
U_s\in M(A)_e=B.
\]
By hypothesis (2),
the map $U\:G\to B$ is norm continuous.
A trivial computation shows that for all $s,t\in G$,
\[
U_{st}=U_s\ad W_s(U_t)=U_s\alpha_s(U_t).
\]
Thus $U$ is an $\alpha$-cocycle,
so we can define an exterior equivalent action
\[
\gamma=\ad U\circ\alpha\:G\curvearrowright B.
\]

Recall
the Pedersen isomorphism
\[
\Theta\:(B\rtimes_\gamma G,\what\gamma)\variso (B\rtimes_\alpha G,\what\alpha)
\]
determined by
\[
\Theta\circ i_B^\gamma=i_B^\alpha
\midtext{and}
\Theta(i_G^\gamma(s))=i_B^\alpha(U_s)i_G^\alpha(s).
\]
On the other hand, by Landstad we may take
\begin{align*}
(B\rtimes_\alpha G,\what\alpha,i_G^\alpha)&=(A,\delta,V)&
B&=A^{\delta,V}\\
(B\rtimes_\gamma G,\what\gamma,i_G^\gamma)&=(A,\delta,W)&
B&=A^{\delta,W},
\end{align*}
with $i_B^\alpha=i_B^\gamma$ both being the inclusion $B\hookrightarrow M(A)$.
In particular, we may assume that $\Theta=\id_A$.
Since we have $C=A^{\delta,W}$ by definition, we conclude that $C=B$,
as desired.
\end{proof}

\begin{ex}
The conditions in \thmref{sufficient} certainly required the generalized fixed-point algebra $B=A^{\delta,V}$ to be unital. However, it is important to keep in mind that it is possible for $B$ to be unital without satisfying condition (2).
For example, let $G=\T$ acting on $C(\T)$ by translation, so that the crossed product $A=C(\T)\rtimes \T$ is the compact operators $\KK$ on $L^2(\T)$.
Then the generalized fixed-point algebra is the unital algebra of multiplication operators $M_\phi$ for $\phi\in C(\T)$.
On bounded sets, the strict topology on $M(\KK)=B(L^2(\T))$ agrees with the strong* topology.
Let $(\phi_n)$ be a sequence of unit vectors in $C(\T)$ whose supports $S_n$ shrink to a point.
Then the multiplication operators $M_{\phi_n}$ go to 0 strong*, but all have norm 1.
Therefore the norm topology of $B=\{M_\phi:\phi\in C(\T)\}$ is strictly stronger than the relative strict topology from $M(\KK)$.
We thank Dana Williams for discussions leading to this example.
\end{ex}

\begin{prop}\label{id tensor}
Let $(B,\alpha)$ be an action such that
\begin{enumerate}
\item $i_B(B)=M(B\rtimes_\alpha G)_e$, and
\item the norm topology on $i_B(B)$ equals the relative strict topology from $M(B\rtimes_\alpha G)$.
\end{enumerate}
Then for any finite-dimensional $C^*$-algebra $A$,
the action
$(A\otimes B,\id\otimes \alpha)$
is strongly Pedersen rigid.
\end{prop}

\begin{proof}
The assumptions on $(B,\alpha)$ are that the dual coaction $\what\alpha$ satisfies the hypotheses of \thmref{sufficient},
and by that same theorem it suffices to show that
dual coaction $\what{\id\otimes\alpha}$ also satisfies those hypotheses.
We have
\[
\bigl((A\otimes B)\rtimes_{\id\otimes\alpha} G,\what{\id\otimes\alpha}\bigr)
=\bigl(A\otimes (B\rtimes_\alpha G),\id\otimes \what\alpha\bigr).
\]
Moreover,
\[
i_{A\otimes B}(A\otimes B)=A\otimes i_B(B).
\]
We want to show that
\[
A\otimes i_B(B)=M\bigl(A\otimes (B\rtimes_\alpha G)\bigr)_e.
\]
Trivially the left side is contained in the right, so
let
\[
m\in M\bigl(A\otimes (B\rtimes_\alpha G)\bigr)_e.
\]
Note that
\[
M\bigl(A\otimes (B\rtimes_\alpha G)\bigr)
=A\otimes M(B\rtimes_\alpha G)
\]
since $A$ is finite-dimensional.
Choose a basis $\{a_1,\dots,a_n\}$ for $A$.
Then $m=\sum_{i=1}^n(a_i\otimes m_i)$ with $m_i\in M(B\rtimes_\alpha G)$,
and we have
\begin{align*}
\sum_{i=1}^n(a_i\otimes m_i\otimes 1)
&=m\otimes 1
=(\id\otimes \what\alpha)(m)
\\&=\sum_{i=1}^n (a_i\otimes \what\alpha(m_i)),
\end{align*}
so because $\{a_1,\dots,a_n\}$ is linearly independent we see that
for each $i$ we have
\[
\what\alpha(m_i)=m_i\otimes 1,
\]
and hence $m_i\in M(B\rtimes_\alpha G)_e=i_B(B)$.
Therefore $m\in A\otimes i_B(B)$, as desired.
\end{proof}

\begin{cor}\label{symmetric}
For any locally compact group $G$,
the trivial action of $G$ on $\C$ is strongly Pedersen rigid.
\end{cor}

\begin{proof}
We will apply \thmref{sufficient}
to prove the equivalent statement that
the canonical coaction $\delta_G$,
which is the dual coaction on the crossed product $C^*(G)=\C\rtimes G$,
is strongly fixed-point rigid.
For this it suffices to prove that
\[
M(C^*(G))_e=\C 1_{M(C^*(G))}
\]
This fact is presumably folklore, but we could not find it in the literature,
so we include the proof.
Obviously the right-hand side is contained in the left.
On the other hand, if $m\in M(C^*(G))_e$,
then
\[
\delta_G(m)=m\otimes 1,
\]
and slicing by $f\in B(G)$ gives
\[
(\id\otimes f)\circ\delta_G(m)=f(e)m.
\]
But the homomorphism $\delta_G$ is symmetric:
\[
\delta_G=\Sigma\circ\delta_G
\]
where $\Sigma$ is the flip automorphism of $C^*(G)\otimes C^*(G)$,
consequently
\[
(\id\otimes f)\circ\delta_G(m)
=(f\otimes\id)\circ\delta_G(m)
=f(m)1.
\]
So, as long as we choose $f$ with $f(e)\ne 0$,
we conclude that $m\in \C 1$.
\end{proof}

Let $H$ be a subgroup of $G$
with finite index.
Let $X=G/H$,
and let $G$ act on $X$ by left translation.
Choose a cross section $x\mapsto \eta_x$ of $X$ in $G$.
It is well-known that the map
$\varphi\:G\times X\to H$
defined by
\[
\varphi(s,x)=\eta_{sx}\inv s\eta_x
\]
is a cocycle
for the action $G\curvearrowright X$,
i.e.,
\[
\varphi(st,x)=\varphi(s,tx)\varphi(t,x).
\]
Note that $\varphi(e,x)=e$ for all $x\in X$, as is the case for all cocycles.
Let $C(X)$ be the commutative $C^*$-algebra of functions on the discrete space $X$, and
$\alpha\:G\curvearrowright C(X)$ be the associated action.
Let $M_X$ denote the matrix algebra on $X$, with matrix units $\{e_{xy}:x,y\in X\}$ characterized by
\begin{align*}
e_{xy}e_{uv}&=\delta_{yu}e_{xv}\\
e_{xy}^*&=e_{yx},
\end{align*}
where $\delta_{yu}$ is the Kronecker delta.

In the following lemma we use the \emph{inflated coaction} $\delta_H$,
and we recall its definition in this special case:
since $H$ has finite index in $G$,
we can regard $C^*(H)$ as a $C^*$-subalgebra of $C^*(G)$,
and $\infl\delta_H$ is defined by the commutative diagram
\[
\xymatrix{
C^*(H) \ar[r]^-{\delta_H} \ar[dr]_{\infl\delta_H}
&M(C^*(H)\otimes C^*(H)) \ar[d]
\\
&M(C^*(H)\otimes C^*(G)),
}
\]
where the vertical arrow is the identity on $C^*(H)$
tensored with the inclusion map
$C^*(H)\hookrightarrow C^*(G)$.

\begin{lem}\label{U}
Define
\[
U=\sum_{x\in X}(e_{xx}\otimes 1\otimes \eta_x)
\in M\bigl(M_X\otimes C^*(H)\otimes C^*(G)\bigr).
\]
Then
$U$ is a cocycle for the coaction $\id\otimes \infl\delta_H$ of $G$.
\end{lem}

\begin{proof}
First,
\begin{align*}
(\id\otimes \delta_G)(U)
&=\sum_{x\in X}
(e_{xx}\otimes 1\otimes \eta_x\otimes \eta_x),
\end{align*}
while
\begin{align*}
&(U\otimes 1)
\bigl((\id\otimes \infl\delta_H)\otimes\id\bigr)(U)
\\&\quad=\sum_{x,y\in X}
(e_{xx}\otimes 1\otimes \eta_x\otimes 1)
(e_{yy}\otimes 1\otimes 1\otimes \eta_y)
\\&\quad=\sum_{x,y\in X}
(e_{xx}e_{yy}\otimes 1\otimes \eta_x\otimes \eta_y)
\\&\quad=\sum_{x\in X}
(e_{xx}\otimes 1\otimes \eta_x\otimes \eta_x).
\end{align*}
For the other axiom of cocycles,
let $x,y\in X$, $c\in C^*(H)$, and $d\in C^*(G)$.
Then
\begin{align*}
&\ad U\circ (\id\otimes\infl\delta_H)(e_{xy}\otimes c)(1\otimes 1\otimes d)
\\&\quad=\ad U\bigl(e_{xy}\otimes \delta_H(c)\bigr)(1\otimes 1\otimes d)
\\&\quad=\sum_{u,v\in X}
(e_{uu}\otimes 1\otimes \eta_u)
(e_{xy}\otimes \delta_H(c)
(e_{vv}\otimes 1\otimes \eta_v\inv)(1\otimes 1\otimes d)
\\&\quad=e_{xy}\otimes (1\otimes \eta_x)\delta_H(c)(1\otimes \eta_y\inv)(1\otimes d)
\\&\quad\in e_{xy}\otimes (1\otimes \eta_x)\delta_H(c)(1\otimes C^*(G))
\\&\quad\subset e_{xy}\otimes (1\otimes \eta_x)(C^*(H)\otimes C^*(G))
\\&\quad\subset e_{xy}\otimes (C^*(H)\otimes C^*(G))
\\&\quad\subset M_X\otimes C^*(H)\otimes C^*(G).
\end{align*}
We have shown that $U$ is an $(\id\otimes \infl\delta_H)$-cocycle.
\end{proof}

Part of the
following
(not involving the coaction)
is a very special case of a theorem of Green
\cite[Corollary~2.10]{gre:structure},
but since our situation is so elementary we give the proof.

\begin{lem}\label{special green}
With the above notation,
define
$\pi\:C(X)\to M(M_X\otimes C^*(H))$ by
\begin{equation}\label{pi}
\pi(f)=\sum_{x\in X}f(x)e_{xx}\otimes 1
\end{equation}
and
$V\:G\to M(M_X\otimes C^*(H))$ by
\begin{equation}\label{V}
V_s=\sum_{x\in X}\bigl(e_{sx,x}\otimes \varphi(s,x)\bigr).
\end{equation}
Then $(\pi,V)$ is a covariant homomorphism of the action $(C(X),\alpha)$,
and the integrated form is an isomorphism
\[
\theta=\pi\times V\:C(X)\rtimes_\alpha G
\variso M_X\otimes C^*(H).
\]
Moreover, the isomorphism $\theta$ transports the dual coaction $\what\alpha$ to the coaction $\delta$ of $G$ on $M_X\otimes C^*(H)$ given
by
\[
\ad U\circ (\id\otimes \infl\delta_H),
\]
\end{lem}

\begin{proof}
Both $C(X)\rtimes_\alpha G$ and $M_X\otimes C^*(H)$ are groupoid $C^*$-algebras:
for the first we use the transformation groupoid $G\times X$ with multiplication
\[
(s,tx)(t,x)=(st,x),
\]
and whose unit space $\{e\}\times X$ we identify with $X$,
and for the second algebra we use the product groupoid $X^2\times H$,
where $X^2$ denotes the full equivalence relation on $X$.
It is folklore that these groupoids are isomorphic,
and we recall how this goes:
recall that we chose 
a cross section $x\mapsto \eta_x$ from $X$ to $G$,
which determined a cocycle
$\varphi\:G\times X\to H$ by
\[
\varphi(s,x)=\eta_{sx}\inv s\eta_x.
\]
This in turn
leads to a groupoid isomorphism
\[
\rho\:G\times X\variso X^2\times H
\]
via
\[
\rho(s,x)=\bigl(sx,x,\varphi(s,x)\bigr),
\]
with inverse given by
\[
\rho\inv(x,y,h)=\bigl(\eta_xh\eta_y\inv,y),
\]
and moreover
$\rho$ is a homeomorphism since $X$ is discrete.
Then $\rho$ determines an isomorphism
\[
\theta\:C(X)\rtimes_\alpha G\variso M_X\otimes C^*(H)
\]
between the groupoid $C^*$-algebras,
given by the integrated form of the 
covariant homomorphism
$(\pi,V)$ defined in \eqref{pi} and \eqref{V}.

The isomorphism $\theta$
transports the dual coaction $\what\alpha$ to a coaction $\delta$.
To compute $\delta$, we first consider an elementary tensor
\[
e_{xy}\otimes h\in M_X\otimes M(C^*(H))
=M\bigl(M_X\otimes C^*(H)\bigr)
\]
for $x,y\in X,h\in H$:
\begin{align*}
\delta(e_{xy}\otimes h)
&=(\theta\otimes\id)\circ\what\alpha\circ\theta\inv(x,y,h)
\\&=(\theta\otimes\id)\circ\what\alpha(\eta_xh\eta_y\inv,y)
\\&=(\theta\otimes\id)\circ\what\alpha\bigl(i_G(\eta_xh\eta_y\inv)i_{C(X)}(\Chi_y)\bigr)
\\&=(\theta\otimes\id)\bigl(i_G(\eta_xh\eta_y\inv)i_{C(X)}(\Chi_y)\otimes \eta_xh\eta_y\inv\bigr)
\\&=(x,y,h)\otimes \eta_xh\eta_y\inv
\\&=e_{xy}\otimes h\otimes \eta_xh\eta_y\inv.
\end{align*}
Then for $f\in C_c(H)$ we have
\begin{align*}
\delta(e_{xy}\otimes f)
&=\int_H f(h)\delta(e_{xy}\otimes h)\,dh
\\&=\int_H f(h)(e_{xy}\otimes h\otimes \eta_xh\eta_y\inv)\,dh
\\&=e_{xy}\otimes \int_H f(h)(h\otimes \eta_xh\eta_y\inv)\,dh
\\&=e_{xy}\otimes (1\otimes \eta_x)\int_H f(h)(h\otimes h)\,dh(1\otimes \eta_y\inv)
\\&=e_{xy}\otimes (1\otimes \eta_x)\infl\delta_H(f)(1\otimes \eta_y\inv).
\end{align*}
On the other hand,
\begin{align*}
&\ad U\circ (\id\otimes \infl\delta_H)(e_{xy}\otimes f)
\\&\quad=U\bigl(e_{xy}\otimes \infl\delta_H(f)\bigr)U^*
\\&\quad=\sum_{u,v\in X}(e_{uu}\otimes 1\otimes \eta_u)
\bigl(e_{xy}\otimes \infl\delta_H(f)\bigr)
(e_{vv}\otimes 1\otimes \eta_v\inv)
\\&\quad=e_{xy}\otimes (1\otimes \eta_x)\infl\delta_H(f)(1\otimes \eta_y\inv).
\end{align*}
Thus by density and continuity we have
\[
\delta=\ad U\circ (\id\otimes \infl\delta_H).
\qedhere
\]
\end{proof}

\begin{thm}\label{no-go orbit}
Let $G$ act transitively on a finite set $X$.
Then the associated action $\alpha$ of $G$ on $C(X)$
is strongly Pedersen rigid.
\end{thm}

\begin{proof}
We will prove the equivalent statement that
the dual coaction $\what\alpha$
is strongly fixed-point rigid,
Now, by \lemref{special green},
$\what\alpha$ is exterior equivalent to the coaction
$\id\otimes\infl\delta_H$ on
$M_X\otimes C^*(H)$.
Since $M_X$ is finite-dimensional,
by \propref{id tensor}
it now suffices to prove that $\infl\delta_H$ is 
strongly fixed-point rigid.
As in the proof of
\corref{symmetric},
we only need to show that
\[
M(C^*(H))_e\subset \C 1_{M(C^*(H))},
\]
where we mean the fixed points in $M(C^*(H))$
relative to the coaction $\infl\delta_H$ of $G$.
So, let $m\in M(C^*(H))_e$,
so that $\infl\delta_H(m)=m\otimes 1_{M(C^*(G))}$.
But then $\infl\delta_H(m)$
must coincide with
the image of
$m\otimes 1_{M(C^*(H))}$
in
$M(C^*(H)\otimes C^*(G))$,
so \corref{symmetric} implies that
$m\in \C 1_{M(C^*(H))}$,
as desired.
\end{proof}

\begin{cor}\label{orbits}
Let $G$ act on a finite set $X$,
and let $\alpha$ be the associated action of $G$ on $C(X)$.
Then the dual coaction $(C(X)\rtimes_\alpha G,\what\alpha)$
is strongly fixed-point rigid.
\end{cor}

\begin{proof}
Since $X$ is a disjoint union of orbits,
$\alpha$ is a finite direct sum of transitive actions on finite sets,
so the corollary follows from \thmref{no-go orbit} and \corref{no-go bundle}.
\end{proof}

We believe that all
actions on finite-dimensional $C^*$-algebras are strongly Pedersen rigid.
However, there is a subtlety that has prevented us from proving a no-go theorem in that generality,
and we explain here:
every action on a finite-dimensional $C^*$-algebra is a direct sum of actions that are transitive on the primitive ideal spaces.
Suppose that
$A$ is finite-dimensional and
$\alpha\:G\curvearrowright A$ is transitive on $X=\prim A$.
Then up to isomorphism
\[
A=M_n\otimes C(X)=\bigoplus_{x\in X}M_n.
\]
Any automorphism of $A$ can be expressed as
a permutation of the copies of $M_n$
followed by an unitary automorphism.
More precisely, for each $s\in G$ we have
\[
\alpha_s=\ad U_s\circ \beta_s,
\]
where
$U_s=(U_s^x)_{x\in X}$
is a tuple of unitary matrices
and
$\beta_s$ just permutes the coordinates in the direct sum of matrices.
The obstruction to $U$ being a $\beta$-cocycle
is a circle-valued two-cocycle $\tau$ on $G$,
which we call the \emph{Mackey obstruction} of the action $\alpha$.

\begin{thm}\label{fin dim}
Let $\alpha$ be an action of $G$ on a finite-dimensional $C^*$-algebra $A$. If all the Mackey obstructions
discussed above
vanish,
then $(A,\alpha)$ is strongly Pedersen rigid.
\end{thm}

\begin{proof}
We continue to use the notation in the discussion preceding the theorem.
By hypothesis,
we can choose the unitaries $U_s$ so that $U$ is a $\beta$-cocycle.
Thus $\alpha$ is exterior equivalent to $\beta$, and
\propref{id tensor} tells us that $\beta$ is strongly Pedersen rigid,
so
$\alpha$ is also strongly Pedersen rigid.
\end{proof}

\begin{cor}\label{no-go unitary}
Every unitary action of $G$ on a finite-dimensional $C^*$-algebra $A$ is strongly Pedersen rigid.
\end{cor}

\begin{proof}
Since the action is unitary, $G$ acts trivially on $\prim A$,
and we are assuming that all the Mackey obstructions vanish,
so this is a special case of \thmref{fin dim}.
\end{proof}

\begin{rem}\label{unitary rem}
In \cite[Proposition~4.6]{koqpedersen} we proved that if $G$ is abelian then every unitary action is strongly Pedersen rigid.
This worked in that much generality because if $\alpha$ is the trivial action of an abelian group on $A$, then the homomorphism $i_G^\alpha\:G\to M(A\rtimes_\alpha G)$ maps into the center,
and hence commutes with $i_G^\beta$ for any other action satisfying $A\rtimes_\alpha G=A\rtimes_\beta G$.
For nonabelian $G$, the best we were able to do is 
\corref{no-go unitary},
which imposes severe restrictions on the action.
\end{rem}


\providecommand{\bysame}{\leavevmode\hbox to3em{\hrulefill}\thinspace}
\providecommand{\MR}{\relax\ifhmode\unskip\space\fi MR }
\providecommand{\MRhref}[2]{%
  \href{http://www.ams.org/mathscinet-getitem?mr=#1}{#2}
}
\providecommand{\href}[2]{#2}

\end{document}